\documentclass{amsart}

\usepackage{xypic}
\input xy
\xyoption{all}
\usepackage{epsfig}
\usepackage{amsthm}
\usepackage{amssymb}
\usepackage{amsmath}
\usepackage{amscd}
\usepackage{color}
\usepackage[T1]{fontenc}

\newcommand{\bg}{\begin{equation}}
\newcommand{\ed}{\end{equation}}
\newcommand{\bga}{\begin{eqnarray}}
\newcommand{\eda}{\end{eqnarray}}

\def\cbdu{\par{\raggedleft$\Box$\par}}

\newtheorem {Theorem}  {Theorem}

\numberwithin{Theorem}{section}

\newtheorem {Lemma}[Theorem]  {Lemma}

\theoremstyle{definition}
\newtheorem{Definition}[Theorem]{Definition}
\theoremstyle{remark}

\newtheorem {Corollary}[Theorem]{\bf Corollary}

\expandafter\chardef\csname pre amssym.def
at\endcsname=\the\catcode`\@ \catcode`\@=11
\def\undefine#1{\let#1\undefined}
\def\newsymbol#1#2#3#4#5{\let\next@\relax
 \ifnum#2=\@ne\let\next@\msafam@\else
 \ifnum#2=\tw@\let\next@\msbfam@\fi\fi
 \mathchardef#1="#3\next@#4#5}
\def\mathhexbox@#1#2#3{\relax
 \ifmmode\mathpalette{}{\m@th\mathchar"#1#2#3}%
 \else\leavevmode\hbox{$\m@th\mathchar"#1#2#3$}\fi}
\def\hexnumber@#1{\ifcase#1 0\or 1\or 2\or 3\or 4\or 5\or 6\or 7\or 8\or
 9\or A\or B\or C\or D\or E\or F\fi}

\font\teneufm=eufm10 \font\seveneufm=eufm7 \font\fiveeufm=eufm5
\newfam\eufmfam
\textfont\eufmfam=\teneufm \scriptfont\eufmfam=\seveneufm
\scriptscriptfont\eufmfam=\fiveeufm

\catcode`\@=\csname pre amssym.def at\endcsname

\newcounter{remark}
\newcommand{\e}{\epsilon}

\renewcommand{\a}{\alpha}

\renewcommand{\b}{\beta}

\newcommand{\s}{\sigma}

\newcommand{\R}{\mathbf{R}}

\def  \R   {{\mathbb R}}

\def  \T   {{\mathbb T}}

\def  \12  {{\frac{1}{2}}}

\def  \tr   {\operatorname{tr}}

\def\build#1_#2^#3{\mathrel{\mathop{\kern 0pt#1}\limits_{#2}^{#3}}}

\begin{document}

\title[Harmonic analysis techniques in regularity problems]{Application of harmonic analysis techniques to regularity problems of dissipative equations}

\author [Mimi Dai]{Mimi Dai}
\address{Department of Mathematics, Stat. and Comp. Sci.,  University of Illinois Chicago, Chicago, IL 60607,USA}
\email{mdai@uic.edu} 

\author [Han Liu]{Han Liu}
\address{Department of Mathematics, Stat. and Comp. Sci.,  University of Illinois Chicago, Chicago, IL 60607,USA}
\email{hliu94@uic.edu}

\thanks{The work of the authors was partially supported by NSF Grant
DMS--1815069.}

\begin{abstract}

We discuss recent advances in the regularity problem of a variety of fluid equations and systems. 
The purpose is to illustrate the advantage of harmonic analysis techniques in obtaining sharper conditional regularity results when compared to classical energy methods.

\bigskip

KEY WORDS: Navier-Stokes equations; complex fluids; Littlewood-Paley decomposition theory; regularity/blow-up criteria.

\hspace{0.02cm}CLASSIFICATION CODE: 76D03, 35Q35.
\end{abstract}

\maketitle

\section{Overview}

In this paper we would like to draw readers' attention to recent progress in the regularity problems of a variety of dissipative equations that describe the motion of certain fluid or complex fluid. The emphasize is to introduce a type of low modes regularity criteria for the Navier-Stokes equations (NSE) and several other partial differential equation models akin to it. In other words, the existence of global regular solutions to these equations can be achieved under a condition that some norm of the low frequency parts of the solutions are controlled. It is also our goal to present a wavenumber splitting framework based on techniques from harmonic analysis, applying which the regularity criteria were obtained. 

We first recall some background of the NSE and various fluid equations. 

\subsection{From the NSE to complex fluid systems}\label{bkgrd}

The  incompressible viscous NSEs, a prototype of a series of fluid equations to appear in this paper, are given by 
\begin{equation}\label{nse}
\begin{cases}
u_t + (u \cdot \nabla) u -\nu \Delta u = -\nabla p, \ \ x\in \mathbb{R}^d, \ t \geq 0, \\
\nabla \cdot u =0, \\
\end{cases}
\end{equation}
where $u$ is the fluid velocity, $p$ represents the pressure, and $\mu$ stands for the constant viscosity coefficient. 

It is know that global regularity for the NSE in three dimensions remains an outstanding open problem. 
On contrast, a Leray-Hopf weak solution was shown to exist globally in time by Leray \cite{L2} and Hopf \cite{H}.
\begin{Definition}
Let $\dot C^\infty_0$ denote the space of test functions that are smooth, with compact support, and divergence free.  Let $\langle, \rangle$ denote the $L^2$-inner product. If a vector valued function $u(t,x) \in C_w(0,T; L^2(\R^3))$ satisfies the weak formulation of system (\ref{nse})
\begin{equation}
\begin{cases}
&\int_0^T (-\langle u, \partial_t \varphi\rangle+\langle (u\cdot \nabla)u, \varphi \rangle + \nu\langle \nabla u, \nabla \varphi \rangle )\mathrm{d}t =0,\\ 
&\int_0^T \langle u, \nabla \varphi \rangle \mathrm{d}s=0, \ \ \forall \varphi \in \dot C^\infty_0([0,T]\times \R^3),
\end{cases}
\end{equation}
then $u$ is called a weak solution of the NSE on $[0,T]$.

A weak solution $u$ to system (\ref{nse}) with $\nu >0$ on $[0, T]$ is a Leray-Hopf solution provided that $u \in L^\infty(0,T; L^2(\R^3)) \cap L^2(0,T; H^1(\R^3))$ and that the following energy inequality
\begin{equation}\label{lyhfs}
\|u(t)\|_2^2 +\nu\int_{t_0}^t \|\nabla u(s)\|_2^2 \mathrm{d}s \leq \|u_0\|_2^2 
\end{equation}
holds for almost every $t_0 \in [0, T]$ and $t \in (t_0, T].$
\end{Definition}
Leray, in his pioneering work \cite{L2}, established global existence of Leray-Hopf solutions for initial data with finite energy.

\begin{Theorem}
Let $u_0 \in L^2_\s(\R^3)$, with $L^2_\s$ being the space of divergence
free $L^2$-functions. There exists a global in time Leray-Hopf solution $u$ to system (\ref{nse}).
\end{Theorem}

It is worthwhile to point out that the uniqueness of Leray-Hopf solutions is unknown. Leray's question, whether a Leray-Hopf solution develops singularity at finite time, does not have an answer yet.
However, there have been several important partial regularity results in forms of conditional uniqueness and regularity criteria. Although there is a vast literature on this topic, we shall only allude to the ones that are relevant to the main purpose of this paper. 
Below we summarize the works of Prodi \cite{P2}, Serrin \cite{S} and Ladyzhenskaya \cite{L1}, as well as that of Escauriaza, Seregin and \v Sverak \cite{ESS}.

\begin{Theorem}[Prodi-Serrin-Ladyzhenskaya]\label{psles}
Let $u$ be a Leray-Hopf solution to system (\ref{nse}) with $u_0 \in L^2_\s(\R^3).$ If $u \in L^\a(0, T; L^\b (\R^3)),$ with $\frac{2}{\a}+\frac{3}{\b}=1$ and $\beta \in (3, +\infty],$ then $u$ is smooth on $[0,T]$. 
\end{Theorem}

\begin{Theorem}[Escauriaza-Seregin-\v Sverak]\label{ESS}
Let $u$ be a Leray-Hopf solution to system (\ref{nse}) with $u_0 \in L^2_\s(\R^3).$ If $u \in L^\infty(0, T; L^3 (\R^3)),$  then $u$ is  smooth on $[0,T]$. 
\end{Theorem}

We note that system (\ref{nse}) is invariant under the following scaling: $$u_{\lambda}(t,x) =\frac{1}{\lambda}u(\frac{t}{\lambda^2}, \frac{x}{\lambda}), \ p_\lambda(t,x) = \frac{1}{\lambda^2}p(\frac{t}{\lambda^2}, \frac{x}{\lambda}).$$ The norms of critical spaces i.e., function spaces invariant under the scaling, are of particular importance for the study of the global regularity problem. Indeed, if $(u, p)$ is a solution to the NSE on $[0,T),$ then $(u_\lambda, p_\lambda)$ is a solution to the NSE on $[0, \lambda^2 T)$ with initial data $\lambda^{-1}u_0(\lambda^{-1} \ \cdot \ );$ hence, it is natural that conditions that guarantee global well-posedness of the system are scaling-invariant. Some examples of critical spaces for the NSE are $$\dot H^{\frac{1}{2}} \hookrightarrow L^3 \hookrightarrow\dot B^{-1+\frac{3}{p}}_{p, \infty}\hookrightarrow BMO^{-1} \hookrightarrow \dot B^{-1}_{\infty, \infty},  \ \ \ p<\infty.$$

The time-space Lebesgue spaces in Theorem \ref{psles} are invariant with respect to the scaling, thus critical. However, a Leray-Hopf solution belongs to $L^\a(0, T, L^\b (\R^3))$ with  $\frac{2}{\a}+\frac{3}{\b}=\frac{3}{2},$ which are supercritical. In fact, one of the essential reasons why the NSE regularity problem is particularly challenging is that all the quantities that are known to be controlled to produce a priori bounds e.g., the energy $\|u\|_{L^2}^2$ etc., are supercritical. 

In the case $\nu=0,$ system (\ref{nse}) reduces to the Euler equations for incompressible, inviscid flows, whose solvability in 3D remains more of a mystery than that of the NSE, since not even the global existence of Leray-Hopf type solutions is known, while local existence of mild solutions was established by Kato \cite{K1}. A classical result concerning the Euler equations is the extensibility criterion in terms of the time-integrability of the vorticity, obtained by Beale, Kato and Majda \cite{BKM}.
\begin{Theorem}[Beale-Kato-Majda]\label{bkm}
Let $\nu=0.$ Suppose that $u_0 \in H^s(\R^3), s\geq 3,$ then there exists a time $T$ depending on $\|u_0\|_{H^3},$ so that system (\ref{nse}) has a solution $u \in C(0, T; H^s(\R^3)) \cap C^1(0, T; H^{s-1}(\R^3)).$ Moreover, $u$ can be extended beyond $T$ if and only if $$\int^{T}_0 \|\nabla \times u(t)\|_\infty \mathrm{d}t < +\infty.$$
\end{Theorem}

Besides the NSE, the surface quasi-geostrophic equation (SQG) sharing certain analogy with the NSE, in this paper we also include several complex fluid models such as the magneto-hydrodynamics system (MHD), the Hall-magneto-hydrodynamics system (Hall-MHD), the nematic LCD system with Q-tensor, and the chemotaxis-Navier-Stokes system. These systems have many features (e.g., scaling properties) analogous to those of the NSE. In particular, for each of the aforementioned systems, certain Prodi-Serrin-Ladyzhenskaya or Beale-Kato-Majda type regularity criterion is known and more recently, improvements in the form of low modes regularity criteria have been made. We shall provide more detailed reviews of these fluid equations in the upcoming sections.

\subsection{Kolmogorov's theory of turbulence}\label{kmgrv}

Seemingly fluctuating and chaotic behaviors that fluid flows can exhibit at times pose a major difficulty to obtaining a complete mathematical theory for the fluid equations such as the NSE.   A turbulent flow may be characterized by eddies of different sizes where the energy cascade takes place. While eddies of larger scales break up into those of smaller scales, kinetic energy is also transferred from larger to smaller scales successively, and finally converted into heat by viscosity, as L. F. Richardson depicted, "Big whirls have little whirls that feed on their velocity,
and little whirls have lesser whirls and so on to viscosity in the molecular sense." 

In 1941, A. N. Kolmogorov formulated a mathematical theory of turbulence \cite{K2}. For a fluid with sufficiently high Reynolds number, Kolmogorov suggested that the turbulent flow is statistically isotropic at small scales, in other words, the statistics of the turbulent flow at small scales is independent of directional conditions. Morevoer, it is postulated that at small scales, the statistics of a turbulent flow is universally and uniquely determined by the rate of energy suppy $\e$ and viscosity $\nu$, while determined solely by $\e$ at large scales. The heuristic is that in the NSE, the norm of $\Delta u$ increases as the scales decreases in the energy cascade, thus eventually the molecular dissipation term $\Delta u$ overwhelms the inertial term $u\cdot \nabla u$. 

Via a dimensional argument, one can infer that the length scale at which the turbulence switches from the inertial range to the dissipation range should be $\ell_\mathrm{d} \sim( \frac{\nu^3}{\e})^\frac{1}{4}$.  Kolmogorov's dissipation wavenumber $\kappa_\mathrm{d},$ below which is the range where viscous effects in a turbulent flow are negligible, is then defined as $$\kappa_\mathrm{d}=\e^\frac{1}{4}\nu^{-\frac{3}{4}} \sim \frac{1}{\ell_\mathrm{d}}.$$ 
The analysis above forms a very import conjecture: the low frequency part below $\kappa_\mathrm{d}$ are essential to describe the flow, while the high frequency part above $\kappa_\mathrm{d}$ are asymptotically controlled by low modes. The wavenumber splitting approach to be introduced later is inspired from this conjecture.



\section{Harmonic analysis tools}\label{mathe}

\subsection{Littlewood-Paley theory}

We recall the Littlewood-Paley decomposition, a tool extensively used in the mathematical study of fluids and waves. To start, we define a family of functions with annular support, $\{\varphi_q(\xi)\}_{q=-1}^\infty$, which forms a dyadic partition of unity in the frequency domain. Let $\lambda_q =2^q, q \in \mathbb{Z}.$ We choose a radial function $\chi \in C^\infty_0(\R^n)$ satisfying
\begin{equation}\notag \chi(\xi)=
\begin{cases}
1, \ \text{for }|\xi| \leq \frac{3}{4}\\
0, \ \text{for }|\xi| \geq 1,\\
\end{cases}
\end{equation}
and define $\varphi(x)=\chi(\frac{\xi}{2})-\chi(\xi)$ and $\varphi_q(\xi)=\begin{cases} \varphi(\lambda_q^{-1}\xi), \text{for }q \geq 0,  \\ \chi(\xi), \ \text{for }q= -1. \end{cases}$ 

Given a vector field $u \in \mathcal{S}^{'},$ its Littlewood-Paley projections are defined as 
\begin{equation}\notag
\begin{cases}
\Delta_{-1} u= u_{-1}=: \mathcal{F}^{-1}(\chi (\xi)\hat u(\xi))=\int \tilde h(y)u(x-y)\mathrm{d}y,\\
\Delta_q u = u_q= : \mathcal{F}^{-1}(\varphi_q(\xi)\hat u(\xi))=\lambda_q^n\int h(\lambda_qy)u(x-y)\mathrm{d}y,\\
\end{cases}
\end{equation}
with $\tilde h =\check \chi$ and $h =\check \varphi.$ 

Thus, the following identity holds in the distributional sense $$u =\sum_{q=-1}^\infty u_q.$$ For convenience, we introduce the following notations $$u_{ \leq Q}= \sum_{q =-1}^Q u_q, \ \ u_{(P, Q]}= \sum_{q= P+1}^Q u_q, \ \  \tilde u_q = \sum_{|p-q| \leq 1} u_p.$$ 
The projections $\{\Delta_q\}_{q=-1}^\infty$ provide us with an elegant tool for frequency localization. An observation is that for a function whose Fourier transform is supported in an annulus e.g., $u_q$, a derivative costs exactly one $\lambda_q.$ More precisely, we have the following norm equivalence
$$\|u\|_{H^s} \sim \Big(\sum_{q \geq -1}^\infty \lambda_q^{2s}\| u_q\|^2_2 \Big)^\frac{1}{2}.$$ For functions with annular support in the frequency domain, e.g., $u_q,$ the following Bernstein's inequality holds, which is applied extensively.
\begin{Lemma}Let $n$ be the space dimension and let $s \geq r \geq 1,$ then
\begin{equation}\notag
\|u_q\|_{r} \lesssim \lambda_q^{n(\frac{1}{r}-\frac{1}{s})} \|u_q\|_s. 
\end{equation}
\end{Lemma}

Finally, the Besov space $B^s_{p, \infty}$ can be defined in a straightforward manner with the Littlewood-Paley projections.
\begin{Definition}The Besov space $B^s_{p, \infty}$ consists of functions $u \in \mathcal{S}^{'}$ such that $$\|u\|_{B^s_{p, \infty}}=:\sup_{q \geq -1} \lambda_q^s\|u_q\|_p <\infty.$$ Here $\| \cdot \|_{B^s_{p, \infty}}$ is the norm of the Besov space. 
\end{Definition} 

We refer readers to the work of Bahouri, Chemin and Danchin \cite{}, as well as that of Grafakos \cite{} for more details about the Littlewood-Paley theory and its applications.

\subsection{The commutator and Bony's paraproduct} 

To deal with the quadratic nonlinear term $(u \cdot \nabla) u$ which makes the NSE intriguing, we seek a way to decompose the product of two functions within the Littlewood-Paley framework. Formally, the product of two distributions $u$ and $v$ can be written as $$uv=\sum_{p,q \geq -1} u_p v_q.$$ In 1981 the para-differential calculus was introduced by J. M. Bony, soon finding its applications in many fields such as the study of nonlinear hyperbolic systems. Here we just introduce Bony's paraproduct, which distinguishes three parts in the product $uv,$ that is,
$$uv= \sum_{q \geq -1} \big(u_{\leq q-2} v_q+ u_q v_{\leq q-2}+\tilde u_q v_q\big).$$ 

Another tool to facilitate the estimation of the convection or inertial terms is the commutator notation
$$[\Delta_q, u_{\leq p-2}\cdot \nabla] v_p = \Delta_q(u_{\leq p-2}\cdot \nabla v_p)-u_{\leq p-2} \cdot \nabla \Delta_q v_p,$$ 
which enjoys the following estimate.
\begin{Lemma}\label{cmm} For $\frac{1}{r_1}=\frac{1}{r_2}+\frac{1}{r_3},$ we have
\begin{equation}\notag
\|[\Delta_q, u_{\leq p-2}\cdot \nabla] v_p\|_{r_1} \lesssim \|v_p\|_{r_2} \sum_{p' \leq p-2}\lambda_{p'}\| u_{p'}\|_{r_3}.
\end{equation}
\end{Lemma}
\begin{proof}
By definition of $\Delta_q$, 
\begin{equation}\notag
\begin{split}
[\Delta_q, u_{\leq p-2}\cdot \nabla] v_p=&\lambda_q^3\int_{\R^3} h(\lambda_q(x-y))(u_{\leq p-2}(y)-u_{\leq p-2}(x))\nabla_y v_p(y) \mathrm{d}y\\
=& -\lambda_q^3\int_{\R^3} \nabla_y h(\lambda_q(x-y))(u_{\leq p-2}(y)-u_{\leq p-2}(x))v_p(y) \mathrm{d}y\\
=& \int_{\R^3} \lambda_q^3|x-y|\nabla_y h(\lambda_q(x-y))\frac{u_{\leq p-2}(x)-u_{\leq p-2}(y)}{|x-y|}v_p(y) \mathrm{d}y.
\end{split}
\end{equation}
By Young's inequality for convolutions, 
\begin{equation}\notag
\begin{split}
\| [\Delta_q, u_{\leq p-2}\cdot \nabla] v_p\|_{r_1}
\leq & \|v_p\|_{r_2}\|u_{\leq p-2}\|_{r_3}\int_{\R^3} |z||\nabla h(z)|\mathrm{d}z\\
\lesssim & \|v_p\|_{r_2}\|u_{\leq p-2}\|_{r_3}.
\end{split}
\end{equation}
\end{proof}
As we shall see in the next section, the commutator, together with the divergence free condition, reveals certain cancellations within the nonlinear interactions.

\section{Low modes regularity criteria for fluid equations}

\subsection{The NSE}
The results of interest in this section are low modes regularity criteria for the NSE, which have improved previously known regularity criteria. In this section, more detailed analysis shall be included as the NSE is the prototypical case for other fluid systems. We shall also discuss more about the results' connection to Kolmogorov's theory of turbulence.

The following result, due to Cheskidov and Shvydkoy \cite{CS2}, is the foremost among a series of low modes regularity criteria of interest in this paper.
\begin{Theorem}[Cheskidov-Shvydkoy]\label{csh}
Let $u$ be a weak solution to system (\ref{nse}) on $[0,T].$ If $u(t)$ is regular on $[0,T)$ and $$\int_0^T \|(\nabla \times u)_{\leq Q(t)}\|_{B^0_{\infty, \infty}} \mathrm{dt} <\infty,$$ then $u(t)$ is regular on $[0,T].$

Here $Q(t)=\log_2 \Lambda(t),$ with the wavenumber $\Lambda(t)$ for the NSE defined as 
\begin{equation}\notag
\Lambda(t)=\min\{\lambda_q: \lambda_p^{-1+\frac{3}{r}}\|u_p\|_r < c_r\nu, \ 2\leq r \leq \infty, \forall p >q, q \in \mathbb{N}\}.
\end{equation}
\end{Theorem} 
This result gives a unified regularity criterion for the NSE and the Euler equations, since in the case of $\nu =0,$ theorem (\ref{csh}) reduces to the Beale-Kato-Majda regularity criterion for the Euler equations.

Later, Cheskidov and Dai \cite{CD} further weakened the above regularity criterion.

\begin{Theorem}[Cheskidov-Dai]\label{chd}
Let $u$ be a weak solution to system (\ref{nse}) on $[0,T]$ such that $u(t)$ is regular on $[0,T).$ If for certain constant $c_r, \ 2 \leq r \leq \infty$ $$\limsup_{q \to \infty }\int^T_{\frac{T}{2}}1_{q \leq Q(t)}\lambda_q\|u_q\|_\infty \mathrm{dt} < c_r,$$ then $u(t)$ is regular on $[0,T].$
\end{Theorem} 

The preludes to the above two results are notable results of regularity criteria in terms of Besov norms. Theorem (\ref{csh}) has improved the following Prodi-Serrin-Ladyzhenskaya criterion extended to Besov spaces, obtained by Kozono, Ogawa and Taniuchi \cite{KOT}.



\begin{Theorem}[Kozono-Ogawa-Taniuchi]\label{kot}
Let $u$ be a weak solution to system (\ref{nse}) such that it is regular on $[0,T)$. If $u \in L^1(0, T; B^{-1}_{\infty, \infty}),$ then $u$ can be extended beyond time $T.$ 
\end{Theorem}

In fact, it can be shown that the condition in theorem (\ref{csh}) is weaker than any Prodi-Serrin-Ladyzhenskaya type condition $u \in L^r(0, T; B^{2/r-1}_{\infty, \infty}).$ Meanwhile, theorem (\ref{chd}) has improved not only theorem (\ref{csh}) but also Planchon's refined Beale-Kato-Majda criterion (see \cite{P1}), stated as follows.

\begin{Theorem}[Planchon]\label{pln}
Let $u \in C(0,T; B^s_{p,q}), s \geq 1+\frac{n}{p}, 1\leq p,q \leq \infty$ be a solution to the Euler equations, that is, system (\ref{nse}) with $\nu=0$. There exists a constant $M_0$ such that $T$ is the maximal existence time iff $$\lim_{\varepsilon \to 0} \sup_{q \geq -1} \int_{T-\varepsilon}^{T}\|\Delta_q(\nabla \times u)\|_\infty \mathrm{d}t\geq M_0.$$
\end{Theorem}

A sketch of the proofs of theorems (\ref{csh}) and (\ref{chd}) starts with considering the Littlewood-Paley projections of the NSE in higher order energy spaces. 
\begin{equation}\label{eng}
\begin{split}
\frac{1}{2}\frac{\mathrm{d}}{\mathrm{d}t}\sum_{q \geq -1}^\infty\lambda_q^{2s}\|u_q\|_2^2 \leq & -\nu\sum_{q \geq -1}^\infty \lambda_q^{2s+2}\|u_q\|_2^2-\sum_{q \geq -1}^\infty\int_{\R^3} \Delta_q(u \cdot \nabla u) u_q \mathrm{d}x\\
=: & -\nu\sum_{q \geq -1}^\infty \lambda_q^{2s+2}\|u_q\|_2^2+ I.
\end{split}
\end{equation}
It then becomes clear that the essential step is to analyze the nonlinear term $(u \cdot \nabla) u$, which translates into term $I$ in (\ref{eng}). As we shall see, the wavenumber splitting approach yields, for $s >\frac{1}{2},$
$$ |I| \lesssim c_r\nu \sum_{q \geq -1}\lambda_q^{2s+2}\|u_q\|_2^2+ Q(t)\|u_{\leq Q(t)}\|_{B^1_{\infty, \infty}}\sum_{q \geq -1}\lambda_q^{2s}\|u_q\|_2^2.$$  

To proceed, $I$ is split into many terms, which allows one to analyze the interactions between different frequencies. Bony's paraproduct decomposition yields 
\begin{equation}\notag
\begin{split}
I= & -\sum_{q \geq -1}\sum_{|p-q| \leq 2}\int_{\R^3} \Delta_q(u_p \cdot \nabla u_{ \leq p-2}) u_q \mathrm{d}x\\
& -\sum_{q \geq -1}\sum_{|p-q| \leq 2}\int_{\R^3} \Delta_q(u_{\leq p-2} \cdot \nabla u_{p}) u_q \mathrm{d}x\\
& -\sum_{q \geq -1}\sum_{|p-q| \leq 2}\int_{\R^3} \Delta_q(u_p \cdot \nabla \tilde u_{p}) u_q \mathrm{d}x\\
=: & I_1+ I_2+I_3. 
\end{split}
\end{equation}

Re-writing the terms using the commutator further reveals certain cancellations. We notice that in the following expression $I_{22}$ vanishes as a consequence of the facts $\sum_{|p-q| \leq 2}\Delta_q u_p =u_q$ and $\nabla \cdot u_{\leq p-2}=0.$
\begin{equation}\notag
\begin{split}
I_2= & -\sum_{q \geq -1}\sum_{|p-q| \leq 2} \int_{\R^3}[\Delta_q, u_{\leq p-2} \cdot \nabla ]u_p u_q\mathrm{d}t\\
&- \sum_{q \geq -1}\sum_{|p-q| \leq 2} \int_{\R^3} u_{\leq q-2} \nabla \Delta_q u_p u_q \mathrm{d}t\\
&- \sum_{q \geq -1}\sum_{|p-q| \leq 2} \int_{\R^3} (u_{\leq p-2}-u_{\leq q-2}) \nabla \Delta_q u_p u_q \mathrm{d}t\\
=:& I_{21}+I_{22}+I_{23}.
\end{split}
\end{equation}

One then utilizes $Q(t)$ to split all the terms above into low modes and high modes, for example
\begin{equation}\notag
\begin{split}
I_{1}= & -\sum_{-1 \leq q \leq Q}\sum_{|p-q| \leq 2}\int_{\R^3} \Delta_q(u_p \cdot \nabla u_{ \leq p-2}) u_q \mathrm{d}x\\
& -\sum_{q> Q}\sum_{|p-q| \leq 2}\int_{\R^3} \Delta_q(u_p \cdot \nabla u_{ \leq Q}) u_q \mathrm{d}x\\
& -\sum_{q> Q}\sum_{|p-q| \leq 2}\int_{\R^3} \Delta_q(u_p \cdot \nabla u_{(Q, p-2]}) u_q \mathrm{d}x\\
=:& I_{11}+I_{12}+I_{13}, 
\end{split}
\end{equation}
and 
\begin{equation}\notag
\begin{split}
I_{21}= & -\sum_{ -1 \leq p \leq Q+2 }\sum_{|p-q| \leq 2} \int_{\R^3}[\Delta_q, u_{\leq p-2} \cdot \nabla ]u_p u_q\mathrm{d}t\\
& -\sum_{p > Q+2}\sum_{|p-q| \leq 2} \int_{\R^3}[\Delta_q, u_{\leq Q} \cdot \nabla ]u_p u_q\mathrm{d}t\\
& -\sum_{p > Q+2}\sum_{|p-q| \leq 2} \int_{\R^3}[\Delta_q, u_{(Q, p-2]} \cdot \nabla ]u_p u_q\mathrm{d}t\\
=:& I_{211}+I_{212}+I_{213}, 
\end{split}
\end{equation}  
where $I_{11}, I_{12}, I_{211}$ and $I_{212}$ are low modes, while $I_{13}$ and $I_{213}$ are high modes. To estimate terms involving the commutator, lemma (\ref{cmm}) is used, while H\"older's and Young's inequalities are used to estimate terms such as $I_{11}, I_{12}$ and $I_{13}.$ It turns out that for low modes and high modes terms the following estimates holds true, respectively.
\begin{equation}\notag
\begin{split}
|I_\text{low modes}| \lesssim & Q(t)\|(\nabla \times u)_{\leq Q(t)}\|_{B^0_{\infty, \infty}} \sum_{q \geq -1}^\infty\lambda_q^{2s}\|u_q\|_2^2,\\
|I_\text{high modes}| \lesssim & c_r \nu \sum_{q \geq -1}^\infty\lambda_q^{2s+2}\|u_q\|_2^2.
\end{split}
\end{equation}

Hence the eventual outcome is a Gr\"onwall type inequality
\begin{equation}\label{grw}
\begin{split}
\frac{\mathrm{d}}{\mathrm{d}t}\sum_{q \geq -1}^\infty\lambda_q^{2s}\|u_q\|_2^2 \lesssim & (-1+c_r) \nu \sum_{q \geq -1}^\infty\lambda_q^{2s+2}\|u_q\|_2^2 \\
& + Q(t)\|u_{\leq Q_(t)}\|_{B^1_{\infty, \infty}}\sum_{q \geq -1}^\infty\lambda_q^{2s}\|u_q\|_2^2,
\end{split}
\end{equation}
where $s > \frac{1}{2},$ which already puts $u$ in a subcritical Sobolev space and $c_r$ can be chosen such that $c_r <1$. By the definition of $\Lambda(t)$ and Bernstein's inequality, one can infer that $$\Lambda \lesssim \|u_Q\|_\infty \lesssim \Lambda \lambda_Q^\frac{1}{2}\|u_Q\|_2 \lesssim \Lambda \lambda_Q^s\|u_Q\|_2.$$ Therefore $Q(t) \lesssim (1+\log \|u\|_{H^s}),$ and inequality (\ref{grw}) becomes
\begin{equation}\notag
\begin{split}
\frac{\mathrm{d}}{\mathrm{d}t}\|u\|_{H^s}^2 \lesssim \|(\nabla \times u)_{\leq Q(t)}\|_{B^0_{\infty, \infty}}(1+\log \|u\|_{H^s})\|u\|_{H^s}^2,
\end{split}
\end{equation}
implying that $u \in L^\infty(0,T; H^s)\cap L^2(0, T; H^{s+1})$ is regular on $[0, T]$ provided that $\|(\nabla \times u)_{\leq Q(t)}\|_{B^0_{\infty, \infty}} \in L^1(0,T),$ which is the condition in theorem (\ref{csh}).

In the proof of theorem (\ref{chd}) the condition $\|(\nabla \times u)_{\leq Q(t)}\|_{B^0_{\infty, \infty}} \in L^1(0,T)$ is weakened via a more delicate analysis. First, by the same procedure as above, one arrives at the following Gr\"onwall type inequality, which slightly differs from (\ref{grw}).
\begin{equation}\label{grn}
\begin{split}
\frac{\mathrm{d}}{\mathrm{d}t}\sum_{p \geq -1}^\infty\lambda_p^{2s}\|u_p\|_2^2 \leq C(\nu,r,s) \sum_{q \leq Q(t)}\lambda_q \|u_q\|_\infty \sum_{p \geq -1}^\infty\lambda_p^{2s}\|u_p\|_2^2.
\end{split}
\end{equation} Assuming that the condition in theorem (\ref{chd}) holds, one can introduce an index $$q^*=\Big\{q:\int_{T/2}^T 1_{q \leq Q(t)}\lambda_q\|u_q\|_\infty \mathrm{d}t < c_r,\ \forall q >q^* \Big\},$$ with $c_r = \e \ln 2/C(\nu, r,s)$ for some small $\e>0.$
 Splitting the sum yields $$\sum_{q \leq Q(t)}\lambda_q \|u_q\|_\infty =\sum_{q \leq q^*}\lambda_q \|u_q\|_\infty+\sum_{q^* < q \leq Q(t)}\lambda_q \|u_q\|_\infty=: f_{\leq q^*}+f_{>q^*}.$$ It is then not difficult to see that, for $t \in [\frac{T}{2}, T]$
\begin{equation}\notag
\begin{split}
\int_{\frac{T}{2}}^t \sum_{q \leq Q(\tau)} \lambda_q\|u_q\|_\infty \mathrm{d}\tau =&\int_{\frac{T}{2}}^t f_{\leq q^*}(\tau) \mathrm{d}\tau+ \int_{\frac{T}{2}}^t f_{>q^*}(\tau) \mathrm{d}\tau
\leq q^*\lambda_{q^*}^\frac{5}{2}\|u_0\|_2 T+\bar Q(t) c_r,
\end{split}
\end{equation} 
where $\bar Q(t)=\sup_{\frac{T}{2} \leq \tau \leq t} Q(\tau).$

By the definition of $\Lambda(t)$ and Bernstein's inequality, one has, for $\Lambda(t) >1$
$$c_r \nu \Lambda^{1-\frac{3}{r}} \leq \|u_Q\|_r \lesssim \Lambda^{\frac{3}{2}-\frac{3}{r}}\|u_Q\|_2,$$
from which one can infer 
$$c_r \nu \Lambda^{s-\frac{1}{2}} \lesssim \Lambda^s\|u_Q\|_2 \lesssim \|u\|_{H^s}.$$ Let $s=\frac{1}{2}+\e,$ the above inequality yields $2^{\e \bar Q(t)} \lesssim \frac{1}{\nu}\|u\|_{H^s}.$ Hence, from inequality (\ref{grn}) one has
\begin{equation}\label{grl}
\begin{split}
\|u(t)\|_{H^s}^2 \leq & \exp \Big(C(\nu, r,s)\int^t_{\frac{T}{2}} (f_{\leq q^*}(\tau)+f_{>q^*}(\tau))\mathrm{d}\tau \Big)\|u(T/2)\|_{H^s}^2\\
\leq & \frac{1}{\nu}\exp (C(\nu, r,s)q^*\lambda{q^*}^\frac{5}{2}\|u_0\|_2 T) \sup_{\frac{T}{2}\leq \tau \leq t}\|u(\tau)\|_{H^s} \|u(T/2)\|_{H^s}^2.\\
\end{split}
\end{equation}
Noticing that the right hand side of (\ref{grl}) is bounded, one can thus conclude theorem (\ref{chd}), which assumes a condition weaker than all the ones listed below. 
\begin{align*}\notag
& \limsup_{q \to \infty }\int^T_{\mathcal{T}_q}\|\Delta_q (\nabla \times u)\|_\infty \mathrm{dt} < c_r,\tag{i} \\ 
& \lim_{\varepsilon \to 0}\limsup_{q \to \infty }\int^T_{T-\varepsilon}\|\Delta_q (\nabla \times u)\|_\infty \mathrm{dt} < c_r,\tag{ii} \\
& \int^T_0\sup_{q\leq Q(t)}\|\Delta_q (\nabla \times u)\|_\infty \mathrm{dt} < \infty,\tag{iii} \\
& \limsup_{q \to \infty}\int^T_0 1_{q\leq Q(t)}\big(\lambda_q^{-1+\frac{3}{r}+\frac{2}{\ell}}\|u\|_r\big)^\ell \mathrm{dt} < \nu^{\ell-1} c_r^\ell,\tag{iv} \\
& \limsup_{q \to \infty}\int^T_{\mathcal{T}_q}\big(\lambda_q^{-1+\frac{3}{r}+\frac{2}{\ell}}\|u\|_r\big)^\ell \mathrm{dt} < \nu^{\ell-1} c_r^\ell,\tag{v} \\
& \lim_{\varepsilon \to 0}\limsup_{q \to \infty}\int^T_{T-\varepsilon}\big(\lambda_q^{-1+\frac{3}{r}+\frac{2}{\ell}}\|u\|_r\big)^\ell \mathrm{dt} < \nu^{\ell-1} c_r^\ell,\tag{vi} \\
& \int^T_0\big(\|u_{\leq Q(t)}\|_{B_{r, \infty}^{-1+\frac{3}{r}+\frac{2}{\ell}}}\big)^\ell \mathrm{dt} < \nu^{\ell-1} c_r^\ell,\tag{vii} \\
&  \limsup_{t \to T^-}\|u(t)-u(T)\|_{B^{-1+\frac{3}{r}}_{r,\infty}}\leq \frac{c_r}{2}, \tag{viii}
\end{align*}
where $2 \leq r \leq \infty, 1 \leq \ell \leq \infty,$ and $\mathcal{T}_q:=\sup\{t \in (\frac{T}{2}, T): Q(\tau)<q, \forall \tau \in (\frac{T}{2}, t)\}.$ Here the conditions (vi), (viii) and (iii), (vii) can be found in Cheskidov and Shvydkoy's works \cite{CS1} and \cite{CS2}, respectively.

We recall the wavenumber with intermittency correction $\kappa_\mathrm{d} = (\e/\nu^3)^\frac{1}{4-s}$ as well as Kolmogorov's theory, according to which the rate of energy dissipation/supply $\e,$ defined as $$\e:= \frac{1}{T}\int^T_0 \|\nabla u\|_2^2 \mathrm{d}t$$ plays a role in both the inertial range and the dissipation range, whereas $\nu$ the viscosity is significant only in the dissipation range. A remarkable feature of $\Lambda(t)$ is that it divides the dissipation range and the inertial range more precisely than $\kappa_\mathrm{d}$ in the sense that $\langle \Lambda \rangle \lesssim \kappa_\mathrm{d},$ where $\langle\Lambda \rangle$ denotes the time average of $\Lambda(t).$ To demonstrate this, we denote $\langle \Lambda \rangle_U =: \frac{1}{T}\int_U \Lambda(t)\mathrm{d}t,$ with $U:= \{t \in [0,T]: \Lambda(t) >1\},$ and calculate as follows:
\begin{equation}\notag
\begin{split}
\langle \Lambda \rangle -1 \leq & \langle \Lambda \rangle_U \leq \bigg(\frac{\langle \Lambda (\nu c_\infty )^\frac{2}{4-\s}\rangle^{4-\s}_U}{\nu^2 c_\infty^2}\bigg)^\frac{1}{4-\s}\\
\leq & \bigg(\frac{\langle \|u_Q\|_\infty^\frac{2}{4-\s}\Lambda^\frac{2-\s}{4-\s} \rangle_U^{4-\s}}{\nu^2 c^2_\infty}\bigg)^\frac{1}{4-\s} \leq \bigg(\frac{\langle \|u_Q\|_\infty^2\Lambda^{2-\s}\rangle_U}{\nu^2 c^2_\infty}\bigg)^\frac{1}{4-\s}.\\
\end{split}
\end{equation}
Here the parameter $\s \in [0,3]$ is the dimension of the set in the Fourier space where dissipation occurs, whose dual $d=3-\s$ is the dimension of the set in the physical space where dissipation occurs, as mentioned in section (\ref{kmgrv}). Choosing any $\s$ such that $\Lambda^{2-\s}\|u_Q\|_\infty^2 \lesssim \Lambda^{2}\|u_Q\|_2^2$ is satisfied, it follows that
$$ \langle \Lambda \rangle -1  \lesssim \bigg(\frac{1}{\nu^3} \frac{\nu}{T}\int_U \Lambda(t)^2\|u_Q\|_2^2 \mathrm{d}t\bigg)^\frac{1}{4-\s} \lesssim \Big(\frac{\e}{\nu^3}\Big)^\frac{1}{4-\s}=\kappa_\mathrm{d}.$$
This observation indicates that the effects of viscosity start to manifest earlier than predicted by the dimensional argument. Moreover, it can be shown that if $\s <\frac{3}{2},$ then $u$ is in fact regular on $[0,T].$

\subsection{The MHD system}

The MHD equations describe the evolution of a system consisting of an electrically conducting fluid and an external magnetic field, influencing each other. It is a model vital to plasma physics, geophysics and several branches of engineering. The following form of the MHD system, in which $u$ is the fluid velocity, $p$ the pressure and $b$ the magnetic field, can be derived from the coupling of the NSE with the Maxwell's equations of electromagnetism. 
\begin{equation}\label{mhd}
\begin{cases}
u_t +u \cdot \nabla u -b \cdot \nabla b +\nabla p = \nu \Delta u,\\
b_t +u \cdot \nabla b -b \cdot \nabla u = \mu \Delta b\\
\nabla \cdot u =0, \ \nabla \cdot b=0,
\end{cases}
\end{equation}
with $x \in \R^3$ and $t \geq 0.$ The constants $\nu$ and $\mu$ are the fluid viscosity coefficient and magnetic resistivity coefficient, respectively. The MHD system and the NSE share many similar features; in fact, the MHD system reduces to the NSE if $b \equiv 0$. System (\ref{mhd}) enjoys the following scaling-invariance - 
\begin{equation}\notag
u_\lambda (t,x)= \frac{1}{\lambda}u(\frac{t}{\lambda^2}, \frac{x}{\lambda}), \ b_\lambda (t,x)= \frac{1}{\lambda}b(\frac{t}{\lambda^2}, \frac{x}{\lambda}), \ p_\lambda (t,x)= \frac{1}{\lambda^2}p(\frac{t}{\lambda^2}, \frac{x}{\lambda})
\end{equation}
solve system (\ref{mhd}) with initial data $(\lambda^{-1}u_0(\lambda^{-1}x), \lambda^{-1}b_0(\lambda^{-1}x)),$ provided that $(u(t,x), b(t,x), p(t,x))$ is a solution corresponding to the initial data $(u_0(x), b_0(x)).$ Global existence of Leray-Hopf type weak  solutions was established by Sermange and Temam \cite{ST}, as well as by Duvaut and Lions \cite{DL2}. For the inviscid case $\nu=0,$ a Beale-Kato-Majda type regularity criterion was due to Caflisch, Klapper and Steele \cite{CKS}.
\begin{Theorem}
Let $(u_0, b_0) \in H^s, s >3.$ Then there exists a solution $(u ,b) \in C(0,T; H^s)\cap C^1(0,T; H^{s-1}),$ which blows up at time $T^*$ iff 
$$\int^{T*}_0 (\|\nabla \times u\|_\infty + \|\nabla \times b\|_\infty)\mathrm{d}t=+\infty.$$ 
\end{Theorem}

A Planchon-type regularity criterion, which has extended the above regularity criterion, was due to Cannone, Miao and Chen \cite{CCM}.
\begin{Theorem}
Let $(u_0, b_0) \in B^s_{p,q}, s > \frac{n}{p}+1, 1 \leq p,q <\infty.$ Suppose that $(u,b) \in C(0,T; B^s_{p,q}) \cap C^1(0,T; B^{s-1}_{p,q})$ is a regular solution to system (\ref{mhd}) on $[0,T).$ Then $(u,b)$ can be extended beyond time $T$ iff there exists a constant $M_0$ such that
$$\lim_{\e \to 0} \sup_{q \geq -1}\int_{T-\e}^T (\|\Delta_q(\nabla \times u)\|_\infty+\|\Delta_q(\nabla \times b)\|_\infty) \mathrm{d}t <M_0.$$
\end{Theorem} 

Consistent with numerical and theoretical observations that the velocity $u$ plays the dominant role in the interactions of $u$ and $b,$ Prodi-Serrin type regularity criteria in terms of only $u$ or $\nabla u$ can be found in the works of He and Xin \cite{HX}, as well as, He and Wang \cite{HW}, which can be summarized as follows.
\begin{Theorem} Let $(u_0, b_0) \in L^2_\s.$ Suppose that $(u,b) \in L^\infty (0,T; H) \cap L^2(0,T; H^2)$ is a regular solution to system (\ref{mhd}). Then $(u,b)$ can be extended beyond time $T$ iff 
\begin{equation}\notag
\begin{split}
& u \in L^\a(0,T; L^\b), \ \frac{2}{\a}+\frac{3}{\b}=1,\ 3 < \b \leq \infty,\\
&\text{or }\nabla u \in L^\a(0,T; L^\b), \ \frac{2}{\a}+\frac{3}{\b}=2,\ 3 < \b \leq \infty.
\end{split}
\end{equation}
\end{Theorem}
The above result was later extended to Besov spaces by Chen, Miao and Zhang \cite{CMZ1, CMZ2}, who also obtained a Beale-Kato-Majda type condition in terms of $u$ only. The following theorems summarize their results.
\begin{Theorem} Let $(u_0,b_0) \in L^2_\s.$ Suppose that $(u,b)$ is a weak solution to system (\ref{mhd}) on $[0,T)$. Then $(u,b)$ is regular on $[0,T]$ iff
$$ u \in L^q(0,T; B^s_{p, \infty}), \ \frac{2}{q}+\frac{3}{p}=1+s,\ \frac{3}{1+s} < p \leq \infty.
$$
\end{Theorem} 
\begin{Theorem}
Let $(u_0, b_0) \in H^s, s > \frac{1}{2}.$ Suppose that $(u,b) \in C(0,T; H^s) \cap C^1(0,T; H^{s+1})$ is a regular solution to system (\ref{mhd}) on $[0,T).$ Then $(u,b)$ can be extended beyond time $T$ iff there exists a constant $M_0$ such that
$$\lim_{\e \to 0} \sup_{q \geq -1}\int_{T-\e}^T \|\Delta_q(\nabla \times u)\|_\infty \mathrm{d}t <M_0.$$
\end{Theorem}

Finally, by the same wavenumber splitting approach as that applied to the NSE, Cheskidov and Dai \cite{CD} proved a regularity criterion in terms of the low modes of $u,$ which poses an improvement to all the previous results listed above.
\begin{Theorem} \label{cd1}
Let $(u,b)$ be a weak solution to system (\ref{mhd}) that is regular on $[0,T).$ Then $(u,b)$ is regular on $[0,T]$ iff
$$\limsup_{q \to \infty}\int_\frac{T}{2}^T 1_{q \leq Q(t)}\lambda_q\|u_q\|_\infty \mathrm{d}t <c_r, \  2 \leq r \leq 6.$$ With $Q(t)=\log_2 \Lambda(t),$ where the wavenumber $\Lambda(t)$ is defined as $$\Lambda(t):= \min\{ \lambda_q: \lambda_p^{-1+\frac{3}{r}}\|u_p(t)\|_r <c_r\nu, \ 2 \leq r \leq 6, \forall p >q, q \in \mathbb{N}\}.$$
\end{Theorem}

Proof of theorem (\ref{cd1}) is essentially along the same line as that of theorem (\ref{chd}). As the energy for system (\ref{mhd}) is $\|u(t)\|_2^2+\|b(t)\|_2^2,$ analyzing the projected equations in higher order energy spaces leads to a Gr\"onwall-type inequality. Compared to the NSE, the MHD system has more terms to control and fewer cancellations to utilize, so the above result is somehow surprising. While the magnetic field $b$ didn't appear in the regularity condition, its presence did restrict the choice of index $r,$ reducing the range of $r$ to $[2, 6]$ from $[2,\infty]$ in the case of the NSE. 

\subsection{The Hall-MHD system}
The following Hall-MHD system differs from the MHD system from the previous section by the extra Hall term $-\nabla \times((\nabla \times b) \times b)$ resulting from replacing the resistive Ohm's law by a more generalized Ohm's law which takes the Hall effect into account.
\begin{equation}\label{hlmhd}
\begin{cases}
u_t +u \cdot \nabla u -b \cdot \nabla b +\nabla p = \nu \Delta u,\\
b_t +u \cdot \nabla b -b \cdot \nabla u-\nabla \times((\nabla \times b) \times b) = \mu \Delta b,\\
\nabla \cdot u =0, \ \nabla \cdot b=0,
\end{cases}
\end{equation}
with $x \in \R^3$ and $t \geq 0.$ The Hall-MHD system accurately models plasma with large magnetic gradients and is useful in the study of the magnetic reconnection process. Leray-Hopf type weak solutions to system (\ref{hlmhd}) satisfy the same energy inequality as that for system (\ref{mhd}), as the additional Hall term vanishes at energy level. Yet, the Hall term annuls the scaling property of the MHD system and renders the equation for the magnetic field $b$ nonlinear. As a result, the low modes regularity criterion for system (\ref{hlmhd}), due to Dai \cite{D1}, has to involve both $u$ and $b,$ in contrast to that for system (\ref{mhd}).
\begin{Theorem}\label{mmd} Let $(u,b)$ be a weak solution to system (\ref{hlmhd}) on $[0,T].$ Assume that $(u,b)$ is regular on $[0,T),$ that is, $(u,b)\in C(0,T; H^s(\R^3)), s > \frac{5}{2}.$ If
$$\int_0^T \|u_{\leq Q_u(t)}(t)\|_{B^1_{\infty, \infty}} +\Lambda_b(t)\|b_{\leq Q_b(t)}(t)\|_{B^0_{\infty, \infty}} \mathrm{d}t < \infty,$$
then $(u,b)$ is regular on $[0,T].$ 
\end{Theorem}
Here $2^{Q_u(t)} = \Lambda_u(t)$ and $2^{Q_b(t)} = \Lambda_b(t),$ with \begin{equation}\notag
\begin{split}
\Lambda_u(t):=& \min\{\lambda_q \geq -1: \lambda_p^{-1}\|u_p(t)\|_\infty < c_0\min\{\nu, \mu\}, \forall p >q\}, \\
\Lambda_b(t):=& \min\{\lambda_q \geq -1: \lambda_{p-q}^\delta \|b_p(t)\|_\infty < c_0\min\{\nu, \mu\}, \forall p >q\},\\
\end{split}
\end{equation} 
where $\lambda^\delta_{p-q}$ represents a kernel with $\delta \geq s.$

Theorem (\ref{mmd}) improves known regularity criteria for system (\ref{hlmhd}), notably the following Prodi-Serrin-Ladyzhenskaya type regularity criterion and its extension into the BMO space, proved by Chae and Lee \cite{CL}.
\begin{Theorem}
Suppose that $(u,b)$ is a weak solution to system (\ref{hlmhd}) on $[0,T]$ that is regular on $[0,T).$ If $(u,b)$ satisfy
\begin{equation}\notag
\begin{split}
& \begin{cases}
u \in L^q(0,T; L^p(\R^3)), \ \frac{3}{p}+\frac{2}{q}\leq 1, p \in (3 ,\infty],\\
\nabla b \in L^r(0,T; L^s(\R^3)), \ \frac{3}{s}+\frac{2}{r}\leq 1,p \in (3 ,\infty];\\
\end{cases}\\
&\text{or } u, \nabla b \in L^2(0,T; BMO(\R^3)),\\
\end{split}
\end{equation}
then $(u,b)$ is regular on $[0,T].$ 
\end{Theorem}

In order to prove theorem (\ref{mmd}), it is essential to tame the Hall term, which involves the strongest nonlinearity in system (\ref{hlmhd}). This could be done thanks to the following new commutators and their corresponding estimates.
\begin{Lemma} Given vector valued functions $F$ and $G$, define the commutators
\begin{equation}\notag
\begin{split}
[\Delta_q, F\times \nabla \times]G=& \Delta_q(F \times (\nabla \times G)) -F\times (\nabla \times G_q),\\
[\Delta_q, (\nabla \times F) \times]G=& \Delta_q(\nabla \times (F \times G)) - (\nabla \times F) \times G_q.
\end{split}
\end{equation}
The following estimates hold true
\begin{align*}\|[\Delta_q, F\times \nabla \times]G\|_r \lesssim \|\nabla F\|_\infty \|G\|_r,\\
\|[\Delta_q, (\nabla \times F) \times]G\|_r  \lesssim \|\nabla F\|_\infty \|G\|_r,\end{align*}
provided that $\nabla \cdot F=0$ and $G$ vanishes at large $|x|$ in $\R^3.$
\end{Lemma}

Indeed, the Hall term is translated into the following term $$H=-\lambda_q^{2s} \int_{\R^3}\Delta_q ((\nabla \times b) \times b) \cdot (\nabla \times b_q)\mathrm{d}x,$$ which can be decomposed using Bony's paraproduct as
\begin{equation}\notag
\begin{split}
H= &-\sum_{q \geq -1} \sum_{|p-q| \leq 2}\lambda_q^{2s} \int_{\R^3}\Delta_q ((\nabla \times b_p) \times b_{\leq p-2}) \cdot (\nabla \times b_q)\mathrm{d}x \\
&-\sum_{q \geq -1} \sum_{|p-q| \leq 2}\lambda_q^{2s} \int_{\R^3}\Delta_q ((\nabla \times b_{\leq p-2}) \times b_p) \cdot (\nabla \times b_q)\mathrm{d}x \\
&-\sum_{q \geq -1} \sum_{p \geq q-2}\lambda_q^{2s} \int_{\R^3}\Delta_q ((\nabla \times b_p) \times \tilde b_p) \cdot (\nabla \times b_q)\mathrm{d}x\\
=:& H_1+H_2+H_3.
\end{split}
\end{equation}
Applying the commutators to the above terms reveals cancellations, for example:
\begin{equation}\notag
\begin{split}
H_1=& \sum_{q \geq -1} \sum_{|p-q| \leq 2}\lambda_q^{2s} \int_{\R^3}[\Delta_q, (b_{\leq p-2} \times \nabla \times] b_p \cdot (\nabla \times b_q)\mathrm{d}x\\
&+\sum_{q \geq -1}\lambda_q^{2s} \int_{\R^3} b_{\leq q-2}\times (\nabla \times b_q) \cdot (\nabla \times b_q) \mathrm{d}x\\
&+\sum_{q \geq -1}\sum_{|p-q| \leq 2}\lambda_q^{2s} \int_{\R^3} (b_{\leq p-2}-b_{\leq q-2})\times (\nabla \times (b_p)_q) \cdot (\nabla \times b_q) \mathrm{d}x\\
=:&H_{11}+H_{12}+H_{13},
\end{split}
\end{equation}
where $H_{12}=0$ due to the property of cross product.  

One then estimates the various terms with the help of commutator estimates, H\"older's, Jensen's and Young's inequalities, as well as the definitions of the wavenumbers, ultimately reaching a Gr\"onwall type inequality which yields theorem (\ref{mmd}).

\subsection{The supercritical SQG equation} In \cite{D2}, the wavenumber spittling method was applied to the supercritical SQG equation, written as follows.
\begin{equation}\label{sqg}
\begin{cases}
\theta_t + u\cdot \nabla \theta+\kappa \Lambda^\a\theta=0,\\
u=R^{\perp} \theta,
\end{cases}
\end{equation}
where $x \in \R^2, t \geq 0,$ with $0 < \a <1, \kappa>0, \Lambda =(-\Delta)^\frac{1}{2}$ and $R^{\perp} \theta =\Lambda^{-1}(-\partial_2 \theta, \partial_1 \theta).$ In system (\ref{sqg}), the scalar function $\theta$ is the potential temperature and the vector valued function $u$ is the fluid velocity. System (\ref{sqg}), which arises from modeling geophysical flows in atmospheric sciences and oceanography, is also a toy model for the 3D Euler equations due to many parallels between the two in their forms and solutions' behaviors. While the subcritical ($1<\a \leq 2$) and critical ($\a=1$) SQG equations are known to be globally well-posed \cite{CV1, CV2, KN, KNV}, the regularity problem for the supercritical SQG equation (\ref{sqg}) remains an intriguing unresolved question. Constantin, Majda and Tabak \cite{CMT} proved a regularity criterion analogous to the one for the Euler equations. 
\begin{Theorem}\label{cmt} Given $\theta_0 \in H^s(\R^2), s\geq 3,$ there exists a unique smooth solution to system (\ref{sqg}) $\theta \in L^\infty(0,T^*; H^s(\R^2)),$ and $T^*$ is the maximal existence time iff 
$$ \int_0^{T^*}\|\nabla^{\perp}\theta(t)\|_\infty \mathrm{d}t = +\infty.$$
\end{Theorem}
A Prodi-Serrin-Ladyzhenskaya type criterion was established by Chae \cite{C}.
\begin{Theorem}\label{chae1} Let $\theta$ be a solution to system (\ref{sqg}). If $$\nabla^{\perp} \theta \in L^q(0,T; L^p(\R^2)), \ \frac{2}{p}+\frac{\a}{q}=\a,\ p \in (\frac{2}{\a}, \infty),$$ 
then $\theta$ is regular on $[0,T].$ 
\end{Theorem}
The Besov space version of the Prodi-Serrin-Ladyzhenskaya type regularity criterion was due to Dong and Pavlovi\'c \cite{DP}.
\begin{Theorem} Let $\theta$ be a weak solution to system (\ref{sqg}). If $$\theta \in L^r(0,T; B^s_{p, \infty}(\R^2)), \ s=\frac{2}{p}+1-\a+\frac{\a}{r}, \  2 \leq p,r < \infty,$$ then $\theta$ is a regular solution on $[0,T].$
\end{Theorem}

Notice that system (\ref{sqg}) enjoys invariance under the scaling transform $$\theta(x,t) \mapsto \theta_\lambda(x,t)=\lambda^{\a-1}\theta(\lambda x, \lambda^\a t),$$ thus some critical spaces are $H^{2-\a}(\R^2),$ $C^{1-\a}(\R^2),$ $L^\infty(\R^2)$ and $B^{1-\a}_{\infty, \infty}(\R^2),$ the last being the largest critical space for system (\ref{sqg}). In particular, the conditions in the theorems listed above are all in terms of critical quantities.

While Leray-Hopf type weak solutions to system (\ref{sqg}) are weak solutions that satisfy the energy inequality $$\|\theta(t)\|_2^2 +2\kappa \int_{t_0}^t \|\nabla \theta(s)\|_2^2\mathrm{d}s \leq \|u(t_0)\|_2^2,$$ the notion of viscosity solutions also comes into play. A weak solution to system (\ref{sqg}) is a viscosity solution if it is the weak limit of of a sequence of solutions to the following systems with $\e \to 0,$
\begin{equation}\notag
\begin{cases}
\theta_t^\e + R^{\perp} \theta^\e\cdot \nabla \theta^\e+\kappa \Lambda^\a\theta^\e=\e \Delta \theta^\e,\\
\theta^\e(x,0)=\theta_0,
\end{cases}
\end{equation}
where $\theta_0 \in H^s(\R^2), s >1.$

Define the wavenumber $\Lambda(t)=2^{Q(t)}$ as $$\Lambda(t)=\min\{\lambda_q: \lambda_p^{1-\a}\|\theta_p(t)\|_\infty <c_0\kappa, \forall p > q \geq 1\},$$ where $c_0$ is some constant. The low modes regularity criterion for the SQG equation states as follows.
\begin{Theorem}\label{md1}
Let $\theta$ be a viscosity solution to system (\ref{sqg}) on $[0,T].$ Assume that $$\int_0^T \|\nabla \theta_{\leq Q(t)}(t)\|_{B^0_{\infty, \infty}}\mathrm{d}t <+\infty,$$
then $\theta$ is regular on $[0,T].$
\end{Theorem}

It can be shown that theorem (\ref{md1}) has improved all of the regularity criteria for system (\ref{sqg}) mentioned before. Its proof is based upon the following regularity result due to Constantin and Wu \cite{CW}.
\begin{Theorem}\label{cnw}
Let $\theta$ be a Leray-Hopf weak solution to system (\ref{sqg}). If $$ \theta \in L^\infty(t_0, t; C^\delta(\R^2)), \ \delta >1-\a, \ 0<t_0< t<\infty,$$
then $\theta \in C^\infty([t_0,t]\times \R^2).$
\end{Theorem} 
Instead of considering the projected equation in $L^2$-based Sobolev spaces, one approaches the problem via the Besov space $B^s_{l,l}.$ Frequency localization yield the following projected equation.
\begin{equation}\notag
\begin{split}
\frac{\mathrm{d}}{\mathrm{d}t}\sum_{q \geq -1}\lambda_q^{sl}\|\theta_q\|_l^l \leq C\kappa \sum_{q \geq -1}\lambda_q^{sl+\a}\|\theta_q\|_l^l+l\sum_{q \geq -1}\lambda^{sl}_q \int_{\R^3} \Delta_q(u\cdot \nabla \theta)|\theta_q|^{l-2} \theta_q \mathrm{d}x.
\end{split}
\end{equation}
Analyzing using the same tools as those used for the NSE, one obtains a Gr\"onwall type inequality which implies that $\theta \in L^\infty(0,T; B^s_{l,l}(\R^2))$ provided that the condition in theorem (\ref{md1}) holds. Choosing $s,l$ such that $0<s <1, sl>2$ and $\frac{\a}{l} <1-s <\a-\frac{2}{l},$ one infers from theorem (\ref{cnw}) and the embedding $B^s_{l,l} \subset C^{0, s-\frac{2}{l}}$ that $\theta$ is in fact a smooth solution.

\subsection{The nematic LCD system with Q-tensor}

Low modes regularity criteria have been established for the following nematic LCD system with Q-tensor in \cite{D3}.
\begin{equation}\label{lcd}
\begin{cases}
u_t +(u\cdot \nabla) u+\nabla p=\nu \Delta u+\nabla \cdot \Sigma(\mathbb{Q}),\\
\mathbb{Q}_t+(u \cdot \nabla)\mathbb{Q}-\mathbb{S}(\nabla u, \mathbb{Q})=\mu \Delta\mathbb{Q}-\mathcal{L}[\partial F(\mathbb{Q})],\\
\nabla \cdot u=0,
\end{cases}
\end{equation}
with $x \in \R^3$ and $t \geq 0.$
System (\ref{lcd}) is a model for nematic liquid crystal flows. Liquid crystals are matters in an intermediate state between the conventionally observed solid and liquid. The nematic phase, in which the rod-like molecules possess no positional order but long-range directional order through self-aligning in an almost parallel manner, is one of the most common liquid crystal phases. Nematics, due to its fluidity and optical properties, are extremely important to liquid crystal displays (LCD). In system (\ref{lcd}), $u$ is the fluid velocity, $p$ the fluid pressure, while the local configuration of the crystal and the ordering of the molecules are represented by the symmetric and traceless Q-tensor $\mathbb{Q}(t,x) \in \R^{3\times 3}_{\text{sym},0}.$ The constants $\nu$ and $\mu$ stand for the fluid viscosity and the elasticity of the molecular orientation field, respectively. In the simplified case 
tensors $\Sigma$ and $\mathbb{S}$ are given by
\begin{equation}\notag
\Sigma(\mathbb{Q})=\Delta\mathbb{Q}\mathbb{Q} - \mathbb{Q} \Delta\mathbb{Q}- \nabla \mathbb{Q}\otimes\nabla \mathbb{Q}, \ \mathbb{S}(\nabla u, \mathbb{Q})=\Omega(u)\mathbb{Q}-\mathbb{Q}\Omega(u)\\
\end{equation}
where $(\nabla \mathbb{Q}\otimes\nabla \mathbb{Q})_{ij}=\partial_i\mathbb{Q}_{\a \b} \partial_j\mathbb{Q}_{\a \b}$ and the skew-symmetric parts of the rate of the stress tensor $\Omega(u)=\frac{1}{2}(\nabla u-\nabla^t u).$ The operator $\mathcal{L},$ which projects onto the space of traceless matrices, is defined as $$\mathcal{L}[\mathbb{A}]=\mathbb{A}-\frac{1}{3}\tr[\mathbb{A}] \mathbb{I},$$ and the bulk potential function $F(\mathbb{Q})$ takes the Landau-de Gennes form $$F(\mathbb{Q})=\frac{a}{2}|\mathbb{Q}|^2+ \frac{b}{3}\tr|\mathbb{Q}|^3+\frac{c}{4}|\mathbb{Q}|^4.$$ 

Paicu and Zarnescu \cite{PZ1, PZ2} proved existence of weak solutions to system (\ref{lcd}). On bounded domains, Abels, Dolzmann and Liu \cite{ADL1, ADL2} proved existence and uniqueness of local-in-time strong solutions subject to various boundary conditions. A Prodi-Serrin-Ladyzhenskaya type regularity condition $$\nabla u \in L^p(0,T; L^q), \ \frac{2}{p}+\frac{3}{q}=2, \ 2 \leq p\leq 3 $$ can be found in \cite{GR}.

For system (\ref{lcd}), the wavenumber $\Lambda(t)=2^{Q(t)}$ can be defined in an almost identical fashion as that for the MHD system: $$\Lambda(t)=\min\{\lambda_q: \lambda_p^{-1+\frac{3}{r}}\| u_p(t)\|_r <c_r\min\{\nu, \mu\}, \forall p>q, q \in \mathbb{N}\}.$$ In the above definition, $c_r$ is a constant depending only on $r \in [2,6).$ Moreover, it turns out that the low modes regularity criteria for the two systems are almost identical as well. 
\begin{Theorem}\label{d11} Let $(u, \mathbb{Q})$ be a weak solution to system (\ref{lcd}) on $[0,T]$ that is also regular on $[0,T).$ Assume that $$\int_0^T \|\nabla u_{\leq Q(t)}(t) \|_{B^0_{\infty, \infty}}\mathrm{d}t <+\infty,$$ 
then $(u, \mathbb{Q})$ is regular on $[0,T].$
\end{Theorem}
A refined result, similar to theorem (\ref{cd1}), states as follows.
\begin{Theorem}\label{d12}
Let $(u,\mathbb{Q})$ be a weak solution to system (\ref{lcd}) that is regular on $[0,T).$ Assume that
$$\limsup_{q \to \infty}\int_\frac{T}{2}^T 1_{q \leq Q(t)}\lambda_q\|u_q\|_\infty \mathrm{d}t <c$$
for some small constant $c,$ then $(u,\mathbb{Q})$ is regular on $[0,T].$  
\end{Theorem} 

The analysis shares much in common with that of system (\ref{mhd}). However, additional new commutator estimates are needed to handle the terms involving $\mathbb{Q}.$ An obvious yet meaningful implication of theorems (\ref{d11}) and (\ref{d12}) is given by the following corollary.
\begin{Corollary} Let $(u,\mathbb{Q})$ be a weak solution to system (\ref{lcd}). Suppose that either the Prodi-Serrin-Ladyzhenskaya type condition $$ u \in L^s(0,T; L^r(\R^3)), \ \frac{2}{s}+\frac{3}{r} = 1, \ 3 <r<6$$
or the Beale-Kato-Majda type condition
$$ \int_0^T \|\nabla \times u (t)\|_\infty \mathrm{d}t <+\infty$$ holds true, then $(u, \mathbb{Q})$ is regular on $[0,T].$
\end{Corollary}

\subsection{The chemotaxis-Navier-Stokes system}
The following chemotaxis-Navier-Stokes system is a coupling of the NSE with the parabolic Keller-Segel system, emerging from the study of aggregation behaviors of chemotactic cells.
\begin{equation}\label{ksnse}
\begin{cases}
n_t+u\cdot \nabla n =\kappa \Delta n -\nabla \cdot(n \nabla c),\\
c_t+u\cdot \nabla c=\mu \Delta c -nc,\\
u_t+(u \cdot \nabla)u+\nabla p=\nu \Delta u-n\nabla\Phi,\\
\nabla\cdot u =0, \ (t,x) \in \R^+ \times \mathbb{T}^3,
\end{cases}
\end{equation}
where $n$ is the cell density, $c$ the concentration of the attractant, $u$ the fluid velocity, $p$ the fluid pressure, and $\nabla \Phi$ a constant vector field. System (\ref{ksnse}) models the situation in which the cells e.g., \textit{Bacillus subtilis}, which are attracted to certain chemical substance e.g., oxygen, swim in sessile drops of water. Lorz \cite{L3} established the existence of local-in-time weak solutions to system (\ref{ksnse}) on bounded domains in 3D. Global existence of weak solutions under more general assumptions was proved via entropy-energy estimates by Winkler \cite{W}. 

In particular, Chae, Kang and Lee \cite{CKL1, CKL2} proved Prodi-Serrin-Ladyzhenskaya type regularity criteria in the following theorem.
\begin{Theorem}\label{ckl} 
Let $(n,c,u)$ be a weak solution to system (\ref{ksnse}) which is regular on $[0,T),$ that is, $(n, c, u) \in L^\infty(0,T; H^{m-1}\times H^m \times H^m(\R^3)).$ If
\begin{equation}\notag
\begin{split}
&\|u\|_{L^q(0,T; L^p(\R^3))}+\|\nabla c\|_{L^2(0,T;L^\infty(\R^3))}<\infty,\\
&\text{or }\|u\|_{L^q(0,T; L^p(\R^3))}+\|n\|_{L^r(0,T;L^s(\R^3))}<\infty
,
\end{split}
\end{equation}
where $\frac{3}{p}+\frac{2}{q}=1,\ 3 < p \leq \infty$ and $\frac{3}{s}+\frac{2}{r} =2,\ \frac{3}{2}< s \leq \infty,$ then $(n,c,u)$ is regular on $[0,T]$.
\end{Theorem}

System (\ref{ksnse}) satisfies the following scaling property: suppose that $(n, c, u)(t,x)$ is a solution to system (\ref{ksnse}) with initial data $(n_0, c_0, u_0)(x),$ then
\[n_\lambda(t, x)=\lambda^2 n(\lambda^2t, \lambda x),\ c_\lambda(t, x)= c(\lambda^2t, \lambda x), \ u_\lambda(t, x)=\lambda u(\lambda^2t, \lambda x)\ \] 
also solves system (\ref{ksnse}) with initial data 
$$n_{\lambda, 0}=\lambda^2n(\lambda x),\ \ c_{\lambda,0}= c(\lambda x),\ \ u_{\lambda, 0}=\lambda u(\lambda x).$$
Obviously, the Sobolev space $\dot H^{-\frac12}\times \dot H^{\frac12}\times\dot H^{\frac32}(\R^3)$ is critical according to the above scaling of the system. Notice that the conditions in theorem (\ref{ckl}) are in terms of scaling-invariant quantities. In addition, a weak solution $(n,c)$ to system (\ref{ksnse}) possesses the following properties:
$$\|n(t)\|_1=\|n_0\|_1, \ \|c(t)\|_\infty \leq \|c_0\|_\infty.$$

The following low modes regularity criterion, proved in \cite{DL1}, seems somehow surprising, given that the wavenumber splitting method was applied to the $c$ equation, which differs from fluid equations in nature. 
\begin{Theorem}\label{dhl}
Let $(n(t), c(t), u(t))$ be a weak solution to (\ref{ksnse}) on $[0,T].$
Assume that $(n(t), c(t), u(t))$ is regular on $[0, T)$ and
\begin{equation}\notag
\int^{T}_0 \|\nabla c_{\leq Q_c(t)}(t)\|^2_{L^ \infty}+\|u_{\leq Q_u(t)}(t)\|_{B^1_{\infty, \infty}} \mathrm{d}t < \infty,  
\end{equation}
then $(n(t), c(t), u(t))$ is regular on $[0, T].$
\end{Theorem}

While the definition of the wavenumber $\Lambda_u(t)$ and the condition on $u$ are the same as those for the NSE in \cite{CD}, to weaken the condition on $\nabla c$ to low modes $\nabla c_{\leq Q_c}$ is rather challenging due to the lack of divergence free condition. This difficulty is reflected by the rather narrow range for the parameter $r$ in the following definition of the wavenumber $\Lambda_c(t):$ 
$$\Lambda_c(t)=\min \{ \lambda_p : \lambda_p^\frac{3}{r}\|c_p(t)\|_r < C_0\min\{\kappa, \mu, \nu\}, \forall p >q, q \in \mathbb{N}\}, \ 3<r <\frac{3}{1-\varepsilon},$$
where $C_0$ and $\varepsilon$ are both positive small constants.

Considering the frequency localized equations of system (\ref{ksnse}) in Sobolev space $(\dot H^s\times \dot H^{s+1} \times \dot H^{s+1})(\T^3)$ and analyzing using the same set of tools as before lead to a Gr\"onwall type inequality from which one can conclude that \begin{gather*} (n,c,u)\in  L^\infty(0,T; \dot H^s\times \dot H^{s+1} \times \dot H^{s+1}(\T^3)),\\ 
(\nabla n, \nabla c, \nabla u) \in  L^2(0,T; \dot H^{s}\times \dot H^{s+1} \times \dot H^{s+1}(\T^3)) 
\end{gather*} for certain $s=-\varepsilon <0.$ While $n$ and $u$ are already in subcritical spaces, one must lift $c$ to a subcritical space by a parabolic bootstrap argument, which could be carried out thanks to the mixed derivative theorem found in the works of Pr\"uss and Simonett \cite{PS}, the fact that $c \in L^\infty(0,T; L^\infty(\T^3))$ as well as the following lemma on the derivative gain for the heat equation.
\begin{Lemma}
Let $u$ be a solution to the heat equation $u_t -\Delta u= f$ on $\T^d, d\geq 2.$ Assume that $u_0 \in H^{\a+1}$ and $f \in L^2(0,T; H^{\a})$ for $\a \in \R,$ then $$u \in L^2 (0,T; H^{\a+2})\cap H^1(0,T;H^\a).$$
\end{Lemma}

It is noteworthy that the first two equations of system (\ref{ksnse}) i.e., the Keller-Segel model, arise from a discipline quite distant from hydrodynamics. The lack of divergence free conditions, along with the entropy functional $\int_{\T^3} n \ln n \mathrm{d}x$ further dissociates the $n$ and $c$ equations from the family of equations of incompressible fluid flows. The fact that the wavenumber splitting method finds its application in a model of population dynamics motivates one to ask the question whether or how one can apply the same harmonic analysis techniques that bear fruits in the realm of fluids to other utterly different models.

\end{document}